\documentclass[12pt,a4paper,british]{amsart}
\usepackage{babel}

\usepackage{url}
\usepackage[all]{xy}
\usepackage{enumerate}
\usepackage{tikz}
\usepackage{tikz-cd}
\usetikzlibrary{arrows,decorations.markings}
\tikzset{commutative diagrams/.cd,arrow style=tikz,diagrams={>=stealth'}}

\newcommand{\rp}[1]{\mathbb{R}P^{#1}}
\newcommand{\stwo}{\mathbb{S}^{2}}
\newcommand{\N}{\mathbb{N}}

\newcommand{\res}[1]{\textit{#1}}

\newcommand{\topp}{\text{Top}}

\newcommand{\bz}{\mathbb{Z}}
\newcommand{\ra}{\longrightarrow}
\renewcommand{\to}{\longrightarrow}
\newcommand{\lra}{\longrightarrow}
\newcommand{\stac}[1]{\stackrel{#1}{\longrightarrow}}
\newcommand{\til}[1]{\widetilde{#1}}

\renewcommand{\ker}[1]{\ensuremath{\operatorname{\text{Ker}}(#1)}}
\newcommand{\im}[1]{\ensuremath{\operatorname{\text{Im}}(#1)}}
\newcommand{\id}{\ensuremath{\operatorname{\text{Id}}}}

\usepackage[
bookmarks=true,
breaklinks=true,
bookmarksnumbered = true,
colorlinks= true,
urlcolor= green,
anchorcolor = yellow,
citecolor=blue,
]{hyperref}                   

\newtheorem{theorem}{Theorem}[section]
\newtheorem{lemma}[theorem]{Lemma}
\newtheorem{proposition}[theorem]{Proposition}

\theoremstyle{definition}

\theoremstyle{remark}
\newtheorem{remark}[theorem]{Remark}

\numberwithin{equation}{section}

\AtBeginDocument{%
\abovedisplayskip=7pt plus 3pt minus 9pt
\abovedisplayshortskip=7pt plus 3pt
\belowdisplayskip=7pt plus 3pt minus 9pt
\belowdisplayshortskip=7pt plus 3pt minus 4pt}

\begin{document}

\title[Embeddings and (v)cd of braid and mapping class groups]{Embeddings and the (virtual) cohomological dimension of the braid and mapping class groups of surfaces}

\author{Daciberg Lima Gon\c{c}alves}
\address{Departamento de Matem\'atica - IME-USP, Caixa Postal~66281~-~Ag.~Cidade de S\~ao Paulo, CEP:~05314-970 - S\~ao Paulo - SP - Brazil}
\curraddr{}
\email{dlgoncal@ime.usp.br}
\thanks{}

\author{John Guaschi}
\address{Normandie Univ., UNICAEN, CNRS, Laboratoire de Math\'ematiques Nicolas Oresme UMR CNRS~6139, 14000~Caen, France}
\curraddr{}
\email{john.guaschi@unicaen.fr}
\thanks{}

\author{Miguel Maldonado}
\address{Unidad Acad\'emica de Matem\'aticas, Universidad Aut\'onoma de Zacatecas, Calzada Solidaridad entronque Paseo a la Bufa, C.P.~98000, Zacatecas, Mexico}
\curraddr{}
\email{mmaldonado@matematicas.reduaz.mx}
\thanks{}

\subjclass[2010]{}


\date{\today}


\begin{abstract}
In this paper, we make use of the relations between the braid and mapping class groups of a compact, connected, non-orientable surface $N$ without boundary and those of its orientable double covering $S$ to study embeddings of these groups and their (virtual) cohomological dimensions. We first generalise results of~\cite{BirmanChillingworth72,GoncalvesGuaschi12} to show that the mapping class group $MCG(N;k)$ of $N$ relative to a $k$-point subset embeds in the mapping class group $MCG(S;2k)$ of $S$ relative to a $2k$-point subset. We then compute the cohomological dimension of the braid groups of all compact, connected aspherical surfaces without boundary, which generalises results of~\cite{GoncalvesGuaschi15}. Finally, if the genus of $N$ is greater than or equal to $2$, we give upper bounds for the virtual cohomological dimension of $MCG(N;k)$.
\end{abstract}

\maketitle
%

\section{Introduction}\label{sec1}

Let $S$ be a compact, connected surface, and let $X=\{ x_{1},\ldots,x_{k} \}$ be a finite (possibly empty) subset of $S$ of cardinality $k\geq 0$. Let $\topp(S;X)$ denote the group of homeomorphisms of $S$ for the operation of composition that leave $X$ invariant. If $S$ is orientable, let $\topp^+(S;X)$ denote the set of elements of $\topp(S;X)$ that preserve orientation. Note that $\topp^+(S;X)$ is a subgroup of $\topp(S;X)$ of index two. We define the \emph{mapping class group $MCG(S;X)$ of $S$ relative to $X$} by:
\begin{equation}\label{eq:mcgtop}
MCG(S;X)=\begin{cases}
\pi_0\topp^+(S;X) & \text{if $S$ is orientable}\\
\pi_0\topp(S;X) & \text{if $S$ is non orientable.}
\end{cases}
\end{equation}
If $S$ is orientable (resp.\ non orientable), the group $MCG(S;X)$ is thus the set of isotopy classes of $\topp^+(S;X)$ (resp.\ $\topp(S;X)$), the isotopies being relative to the set $X$. If $f\in \topp^+(S;X)$ (resp.\ $\topp(S;X)$) then we let $[f]$ denote its mapping class in $MCG(S;X)$. Up to isomorphism, $MCG(S;X)$ only depends on the cardinality $k$ of the subset $X$, and we shall often denote this group by $MCG(S;k)$. If $X$ is empty, then we omit it from the notation, and shall just write $\topp(S)$, $\topp^{+}(S)$, $MCG(S)$ etc. The mapping class group may also be defined in other categories (PL category, smooth category), the groups obtained being isomorphic~\cite{Boldsen09}. The mapping class group has been widely studied from different points of view~--~see~\cite{Birman75,F,FM,I,P} for example. 

If $k\in \N$, there is a surjective homomorphism $\sigma\colon\thinspace \topp(S;X) \to \Sigma_{k}$ defined by $f(x_{i})=x_{(\sigma(f))(i)}$ for all $f\in \topp(S;X)$ and $i\in \{ 1,\ldots,k \}$, where $\Sigma_{k}$ is the symmetric group on $k$ letters. This homomorphism induces a surjective homomorphism from $MCG(S;X)$ to $\Sigma_{k}$ that we also denote by $\sigma$, and its kernel is called the \emph{pure mapping class group of $S$ relative to $X$}, denoted by $PMCG(S;X)$. So we have a short exact sequence:
\begin{equation}\label{eq:sigmamcg}
1 \longrightarrow PMCG(S;X) \longrightarrow MCG(S;X) \stackrel{\sigma}{\longrightarrow} \Sigma_{k} \longrightarrow 1,
\end{equation}
and if $f\in \topp(S;X)$ then $[f]\in PMCG(S;X)$ if and only if $f$ fixes $X$ pointwise. 

Mapping class groups are closely related to surface braid groups. If $k\in \N$, the \res{$k$\textsuperscript{th} ordered configuration space} $F_k(S)$ of $S$ is the set of all ordered $k$-tuples of distinct points of $S$ that we may consider as a subspace of the $k$-fold product of $S$ with itself. The group $\Sigma_k$ acts freely on $F_k(S)$ by permuting coordinates, and the associated quotient is the \res{$k$\textsuperscript{th} unordered configuration space} of $S$, denoted by $D_k(S)$. The fundamental group $\pi_{1}D_k(S)$ (resp.\ $\pi_{1}F_k(S)$), denoted by $B_{k}(S)$ (resp.\ $P_{k}(S)$) is the \emph{braid group} (resp.\ \emph{pure braid group}) of $S$ on $k$ strings~\cite{FaN,FoN}, and $B_{k}(S)$ and $P_{k}(S)$ are related by a short exact sequence similar to that of~(\ref{eq:sigmamcg}). It is well known that $B_k(\mathbb{R}^{2})$ (resp.\ $P_{k}(\mathbb{R}^{2})$) is isomorphic to \res{Artin's braid group} (resp.\ \res{Artin's pure braid group}) on $k$ strings. If $\pi:\til{S}\ra S$ is a $d$-fold covering map, where $d\in \N$, then for all $k\geq 1$, there is a continuous map:
\begin{equation}\label{eq:liftconfig}
\varphi_k: D_k(S)\lra D_{kd}(\til{S})
\end{equation}
defined by $\varphi_k(A)=\pi^{-1}(A)$ for all $A\in D_k(S)$, and the induced homomorphism $\varphi_{k\#}: B_k(S)\ra B_{kd}(\til{S})$ on the level of fundamental groups is injective~\cite{GoncalvesGuaschi12}. Taking $\pi: \stwo\ra \rp{2}$ to be the standard $2$-fold (universal) covering, where $\stwo$ is the $2$-sphere and $\rp{2}$ is the real projective plane, this result was then applied in~\cite{GoncalvesGuaschi12} to classify the isomorphism classes of the finite subgroups of $B_{k}(\rp{2})$, and to show that $B_{k}(\rp{2})$ and $MCG(\rp{2};k)$ are linear groups for all $k\in \N$. 

If $g\geq 0$ (resp.\ $g\geq 1$), let $S=S_{g}$ (resp.\ $N=N_{g}$) be a compact, connected orientable (resp.\ non-orientable) surface of genus $g$ without boundary. In the non-orientable case, $g$ is the number of projective planes in a connected sum decomposition. If $g\geq 1$ and $k\geq 0$, the orientable double covering $\pi:S_{g-1}\ra N_g$ induces a homomorphism $\phi_k:MCG(N_{g};k)\ra MCG(S_{g-1};2k)$, where the $2k$-point subset of marked points in $S$ is equal to the inverse image by $\pi$ of a $k$-point subset of $N$. One of the main aims of this paper is to generalise the injectivity result of~\cite{GoncalvesGuaschi12} to this homomorphism. If $k=0$ and $g\geq 3$, it was shown in~\cite{BirmanChillingworth72,HT09}
 that there exists an injective homomorphism $\phi:MCG(N_g) \to MCG(S_{g-1})$, and that via $\phi$, $MCG(N_g)$ may be identified with the subgroup of $MCG(S_{g-1})$ that consists of isotopy classes of symmetric homeomorphisms. In Section~\ref{sec:embeddings}, we show that a similar result holds for all $k\geq 0$, namely that the homomorphism $\phi_k$ induced by $\pi$ is injective, and that via $\phi_k$, $MCG(N_g; k)$ may be identified with the subgroup of $MCG(S_{g-1}; 2k)$ that consists of isotopy classes of symmetric homeomorphisms.

%

\begin{theorem}\label{th:braidmcg}
Let $k,g\in \mathbb{N}$, let $N=N_{g}$ be a compact, connected, non-orientable surface of genus $g$ without boundary, and let $S$ be its orientable double covering. 
Then the homomorphism $\phi_k:MCG(N;k)\ra MCG(S;2k)$ induced by the covering $\pi:S\ra N$ is injective. 
Further, if $g\geq 3$ then we have a commutative diagram of the following form:
\begin{equation}\label{eq:mcgdiag}
\begin{tikzcd}
1\arrow{r} &B_k(N)\arrow{r} \arrow[hook]{d}{\varphi_{k\#}} & MCG(N;k)\arrow{r}{\tau_{k}}\arrow[hook]{d}{\phi_k} & MCG(N)\arrow[hook]{d}{\phi} \arrow{r} &1\\
1\arrow{r} &B_{2k}(S)\arrow{r} &MCG(S;2k) \arrow{r}{\til{\tau}_{k}} &MCG(S) \arrow{r} &1,
\end{tikzcd}
\end{equation}
where $\tau_{k}$ and $\til{\tau}_{k}$ are the homomorphisms induced by forgetting the markings on the sets of marked points.
\end{theorem}

Note that in contrast with~\cite{BirmanChillingworth72}, Theorem~\ref{th:braidmcg} holds for all $g\geq 1$, not just for $g\geq 3$. As we recall in Remark~\ref{rem:notinjective}, the result of~\cite{BirmanChillingworth72} does not hold if $g=2$. The proof of this exceptional case $g=2$ in Theorem~\ref{th:braidmcg} will turn out to be the most difficult, in part due to the non injectivity of $\phi:MCG(N_2) \to MCG(S_{1})$.

In Section~\ref{sec:vcd}, we compute the cohomological dimension of the braid groups of all compact surfaces without boundary different from $\stwo$ and $\rp{2}$, and we give an upper bound for the virtual cohomological dimension of the mapping class group $MCG(N_{g};k)$ for all $g\geq 2$ and $k\geq 1$. Recall that the \emph{virtual cohomological dimension} $\operatorname{\text{vcd}}(G)$ of a group $G$ is a generalisation of the cohomological dimension $\operatorname{\text{cd}}(G)$ of $G$, and is defined to be the cohomological dimension of any torsion-free finite index subgroup of $G$~\cite{Br}. In particular, if $G$ is a braid or mapping class group of finite (virtual) cohomological dimension then the corresponding pure braid or mapping class group has the same (virtual) cohomological dimension. In~\cite[Theorem~5]{GoncalvesGuaschi15}, it was shown that if $k\geq 4$ (resp. $k\geq 3$) then $\operatorname{\text{vcd}}(B_{k}(\stwo))=k-3$ (resp.\ $\operatorname{\text{vcd}}(B_{k}(\rp{2}))=k-2$). These results are generalised in the following theorem, the proof being a little more straightforward due to the fact that the braid groups of $S_{g}$ and $N_{g+1}$ are torsion free if $g\geq 1$.

\begin{theorem}\label{th:cdbraid}
Let $g,k\geq 1$. Then:
\begin{equation*}
\operatorname{\text{cd}}(B_{k}(S_{g}))=\operatorname{\text{cd}}(P_{k}(S_{g}))=\operatorname{\text{cd}}(B_{k}(N_{g+1}))=\operatorname{\text{cd}}(P_{k}(N_{g+1}))= k+1.
\end{equation*}
\end{theorem}

We then turn our attention to the mapping class groups. If $2g+k>2$, J.~Harer proved that~\cite[Theorem~4.1]{Harer86}:
\begin{equation*}
\operatorname{\text{vcd}}(MCG(S_g;k))= \begin{cases}
4g-5 & \text{if $k=0$}\\
k-3 & \text{if $g=0$ and $k>0$}\\
4g+k-4 & \text{if $g,k>0$.}
\end{cases}
\end{equation*}
Using the embedding of $MCG(N_{g})$ in $MCG(S_{g-1})$ given in~\cite{BirmanChillingworth72}, G.~Hope and U.~Tillmann showed that if $g\geq 3$ then $\operatorname{\text{vcd}}(MCG(N_g))\leq 4g-9$~\cite[Corollary~2.2]{HT09}. Using~\cite[Theorem~5]{GoncalvesGuaschi15}, it was shown in~\cite[Corollary~6]{GoncalvesGuaschi15} that if $k\geq 4$ (resp. $k\geq 3$) then $\operatorname{\text{vcd}}(MCG(\stwo;k))=k-3$ (resp.\ $\operatorname{\text{vcd}}(MCG(\rp{2};k))=k-2$). In the case of $\stwo$, we thus recover the results of Harer. In the second part of Section~\ref{sec:vcd}, we generalise some of these results to non-orientable surfaces with marked points. 

\begin{proposition}\label{prop:vcdmap}  
Let $k>0$. The mapping class groups $MCG(N_g;k)$ and $PMCG(N_g;k)$ have the same (finite) virtual cohomological dimension. Further:
\begin{enumerate}[(a)]
\item $\operatorname{\text{vcd}}(MCG(N_{2};k))=k$.
\item if $g\geq 3$ then $\operatorname{\text{vcd}}(MCG(N_g;k))\leq 4g+k-8$.
\end{enumerate}
\end{proposition}

This paper is organised as follows. In Section~\ref{sec:odc}, we recall some definitions and results about orientation-true mappings and the orientable double covering of non-orientable surfaces, we describe how we lift an element of $\topp(N)$ to one of $\topp^{+}(S)$ in a continuous manner, and we show how this correspondence induces a homomorphism from the mapping class group of a non-orientable surface to that of its orientable double covering (Proposition~\ref{prop:hopetillmann}). In Section~\ref{sec:embeddings}, we prove Theorem~\ref{th:braidmcg} using long exact sequences of fibrations  involving the groups that appear in equation~(\ref{eq:mcgtop}) and the corresponding braid groups~\cite{Birman69,Scott}. In most cases, using~\cite{GoncalvesGuaschi12,Hamstrom66}, we obtain commutative diagrams of short exact sequences, and the conclusion is obtained in a straightforward manner. The situation is however much more complicated in the case where $N$ is the Klein bottle and $S$ is the $2$-torus $T$, due partly to the fact that the exact sequences that appear in the associated commutative diagrams are no longer short exact. This case requires a detailed analysis, notably in the case $k=1$. In Section~\ref{sec:vcd}, we prove Theorem~\ref{th:cdbraid} and Proposition~\ref{prop:vcdmap}. Finally, in an Appendix, we provide presentations of 
$P_{2}(T)$ and $B_{2}(T)$ that we use in the proof of Theorem~\ref{th:braidmcg}.

\subsubsection*{Acknowledgements}

The first-named author was partially supported by FAPESP-Funda\c c\~ao de Amparo a Pesquisa do
Estado de S\~ao Paulo, Projeto Tem\'atico Topologia Alg\'ebrica, Geom\'etrica   2012/24454-8. He and the second-named author would like to thank 
the `R\'eseau Franco-Br\'esilien en Math\'ematiques' for financial support for their respective visits to the Laboratoire de Math\'ematiques Nicolas Oresme UMR CNRS~6139, Universit\'e de Caen Normandie, from the 2\textsuperscript{nd} to the 15\textsuperscript{th} of September 2016, and to the Instituto de Matem\'atica e Estat\'istica, Universidade de S\~ao Paulo, from the 9\textsuperscript{th} of July to the 1\textsuperscript{st} of August 2016. The third-named author was supported by the Consejo Nacional de Ciencia y Tecnolog\'{i}a to visit the Laboratoire de Math\'ematiques Nicolas Oresme UMR CNRS~6139, Universit\'e de Caen Normandie, from the 1\textsuperscript{st} of February 2014 to the 31\textsuperscript{st} of January 2015. During that period, he was also partially supported by the Laboratorio Internacional Solomon Lefschetz (LAISLA). 

\section{Orientation-true mappings and the orientable double covering}\label{sec:odc}


Let $M,N$ be manifolds with base points $x_0,y_0$, respectively. Following~\cite{O}, a pointed map $f:(M,x_0)\ra (N,y_0)$ is called {\it orientation true} if the induced homomorphism
$$
f_*:\pi_1(M,x_0)\lra \pi_1(N,y_0)
$$
sends orientation-preserving (resp.\ orientation-reversing) elements to orientation-preserving (resp.\ orientation-reversing) elements. In other words, the map $f$ is orientation-true if for all $\alpha\in \pi_1(M,x_0)$, either $\alpha$ and $f_*(\alpha)$ are both orientation preserving or they are are both orientation reversing. In the case of a branched covering $f:M \ra N$ it follows by \cite[Proposition 1.4]{GKZ} that $f$ is orientation-true. In what follows, we will be interested in the case of maps between surfaces. We will start with the case of non-orientable surfaces.

\begin{lemma}\label{lem:l2}
Every homeomorphism $f:N\ra N$ of a non-orientable surface is orientation true. Consequently, the subgroup of all orientation-preserving loops is invariant with respect to $f_*$.
\end{lemma}

Lemma~\ref{lem:l2} is an obvious consequence of~\cite[Proposition~1.4]{GKZ}, but we shall give a more direct proof.

\begin{proof}[Proof of Lemma~\ref{lem:l2}]
Let $x_0\in N$ be a base point and consider $\alpha\in \pi_1(N,x_0)$, which we represent by a loop $h:S^1\ra N$. We fix a local orientation $O_1$ at $x_0$, and we consider the orientation on $N$ induced by $f$, denoted by $O_2$. Moving the orientation $O_1$ along the path $h$ by a finite sequence of small open sets $U_i$, the images $f(U_i)$ may be used to transport the local orientation $O_2$ along the path $f\circ h$.

Consider the maps $h,f\circ h:S^1\ra N$ and the tangent bundle $TN$ of $N$. Pulling back by these maps, we obtain bundles over $S^1$ that are homeomorphic, and that are the trivial bundle if the loop is orientation-preserving or the twisted bundle if the loop is orientation-reversing.
\end{proof}

Let $g\geq 1$, and let $N_{g}$ be as defined in the introduction. The unique orientable double covering of $N_g$ may be constructed as follows  (see~\cite{Lima} for example). Let the orientable surface $S_{g-1}$ be embedded in $\mathbb{R}^{3}$ in such a way that it is invariant under the reflections with respect to the $xy$-, $yz$- and $xz$-planes. Consider the involution $J:S_{g-1}\ra S_{g-1}$ defined by $J(x,y,z)=(-x,-y,-z)$. The orbit space $S_{g-1}/\langle J\rangle$ is homeomorphic to the surface $N_g$, and the associated quotient map $\pi:S_{g-1}\ra N_g$ is a double covering.

To simplify the notation, from now on, we will drop the subscripts $g$ and $g-1$ from $N$ and $S$ respectively unless there is risk of confusion, so $S$ will be the orientable double covering of $N$. As indicated previously, the map $\pi$ gives rise to a map on the level of configuration spaces that induces an injective homomorphism $B_k(N)\ra B_{2k}(S)$ of the corresponding braid groups, and this allows us to study the braid groups of a non-orientable surface in terms of those of its orientable double covering~\cite{GoncalvesGuaschi12}.

The following result is an immediate consequence of Lemma~\ref{lem:l2} and is basically contained in~\cite[Key-Lemma~2.1]{HT09}.
 

\begin{lemma}\label{lem:th1}
Let $f:N\ra N$ be a homeomorphism of a non-orientable surface, and let $\pi:S\ra N$ be the orientable double covering. Then $f$ admits a lift, and the number of lifts is exactly two.
\end{lemma}

\begin{proof}
Using Lemma~\ref{lem:l2}, we know that $f_{\#}|_{\pi_1(S)}(\pi_1(S))\subset  \pi_1(S)$, where we identify $\pi_1(S)$ with a subgroup of $\pi_1(N)$. By basic properties of covering spaces, the map $f$ lifts to a map $f':S \longrightarrow S$, and since $S\longrightarrow N$ is a double covering, there are two lifts. \end{proof}

The lifts of $f$ are in one-to-one correspondence with the group $\langle J\rangle$ of deck transformations. There is thus a natural way to choose a lift in a continuous manner simply by choosing $\tilde f$ to be the lift of degree $1$ (the other lift is of degree $-1$ since $J$ is of degree $-1$). Let $\rho:\topp(N)\ra \topp^{+}(S)$ denote this choice of lift. We may use Lemma~\ref{lem:th1} to compare mapping class groups of orientable and non-orientable surfaces.

\begin{proposition}\label{prop:hopetillmann}
Let $N$ be a non-orientable surface, and let $S$ be its orientable double covering. Then there is a homomorphism $\phi: MCG(N)\ra MCG(S)$ such that the following diagram commutes:
\begin{equation}\label{eq:commrho}
\begin{gathered}
\xymatrix{\topp(N)\ar[r]^{\rho}\ar[d]&\topp^{+}(S)\ar[d]\\MCG(N)\ar[r]^{\phi}&MCG(S).}
\end{gathered}
\end{equation}
Further, if the genus of $N$ is greater than or equal to $3$ then $\phi$ is injective.
\end{proposition}

\begin{proof}
If $f,g\in \topp(N)$ are isotopic homeomorphisms then an isotopy between them lifts to an isotopy between the orientation-preserving homeomorphisms $\rho(f)$ and $\rho(g)$ of $S$. This proves the first part of the statement. 
 For the second part, suppose that the genus of $N$ is greater than or equal to $3$. Let $MCG^{\pm}(S)$ denote the extended mapping class group of $S$ consisting of the isotopy classes of all homeomorphisms of $S$, and let $C\!\left\langle J\right\rangle$ be the subgroup of $MCG^{\pm}(S)$ defined by:
\begin{equation*}
C\!\left\langle J\right\rangle=\left\{ \left. [f]\in MCG^{\pm}(S) \,\right\rvert\,\text{$f\in \topp(S)$, and there exists $f'\simeq f$ such that $f'\circ J=J\circ f'$}\right\},
\end{equation*}
where $f'\simeq f$ means that $f'$ is isotopic to $f$. By~\cite{BirmanChillingworth72}, $MCG(N)\cong C\!\left\langle J\right\rangle\!/\!\left\langle [J]\right\rangle$. Let $\pi\colon C\!\left\langle J\right\rangle \to C\!\left\langle J\right\rangle\!/\!\left\langle[ J]\right\rangle$ denote the quotient map. Let $f\in \topp(N)$. Then $\phi([f])=[\rho(f)]$ by diagram~(\ref{eq:commrho}). By Lemma~\ref{lem:th1} and the comment that follows it, $f$ admits exactly two lifts, $\rho(f)$ and $J\circ \rho(f)$, the first (resp.\ second) preserving (resp.\ reversing) orientation. Now $\rho(f)\circ J$ is an orientation-reversing lift of $f$, from which we conclude that $J\circ \rho(f)=\rho(f)\circ J$. Hence $[\rho(f)]\in C\!\left\langle J\right\rangle\cap MCG(S)$, and so $\im{\phi}\subset C\!\left\langle J\right\rangle\cap MCG(S)$. On the other hand, $\pi\left\lvert_{C\left\langle J\right\rangle\cap MCG(S)}\right. \colon C\!\left\langle J\right\rangle\cap MCG(S) \to MCG(N)$ is an isomorphism using the proof of~\cite[Key-Lemma~2.1]{HT09}. It then follows
that $\pi\circ\phi=\operatorname{Id}_{MCG(N)}$, in particular $\phi$ is injective.
\end{proof}

In the case of surfaces with marked points, there is another continuous way to choose a lift. For example, given a finite subset  $X=\{x_1,\ldots ,x_k\}$ of $N$, denote its preimage under $\pi$ by:
\begin{equation*}
\til{X}=\{x_1',x_1'',x_2',x_2'', \ldots,x_k',x_k''\},
\end{equation*}
where $\{ x_i',x_i'' \}=\pi^{-1}(x_{i})$ for all $1\leq i\leq k$. Pick a subset of $\til{X}$ that contains exactly one element of $\{ x_i',x_i'' \}$, denoted by $\overline x_i$, for all $1\leq i\leq k$. If $f\in \topp(N;X)$ then the restriction of $f$ to $X$ is a given permutation of $X$, and if $j\in \{ 1,\ldots,k\}$ is such that $f(x_1)=x_j$, we may define $\til f$ to be the unique lift of $f$ such that $\til f(\overline x_1)=\overline x_j$. The two choices for $\overline x_j$ correspond to the two possible lifts of $f$.

\begin{remark}\label{rem:notinjective}
Suppose that the genus of $N$ is $2$, in which case $N$ is the Klein bottle $K$. If $T$ is the torus, the homomorphism $\phi\colon MCG(K) \to MCG(T)$ is not injective, and it may be described as follows. Set $\pi_1(K)=\left\langle \left. \alpha, \beta \,\right\lvert\, \alpha\beta\alpha\beta^{-1} \right\rangle$. The group $MCG(K)$ is isomorphic to $\mathbb{Z}_2\oplus \mathbb{Z}_2$~\cite{Lick},
and its elements are given by the mapping classes of the following automorphisms: 
\begin{equation}\label{eq:eltsmcgK}
\begin{cases}
\alpha \longmapsto \alpha\\
\beta \longmapsto \beta, 
\end{cases} \quad \begin{cases}
\alpha \longmapsto \alpha\\
\beta \longmapsto \alpha\beta, 
\end{cases} \quad
\begin{cases}
\alpha \longmapsto \alpha\\
\beta \longmapsto \beta^{-1}, 
\end{cases} \quad
\begin{cases}
\alpha \longmapsto \alpha\\
\beta \longmapsto \alpha\beta^{-1}.
\end{cases}
\end{equation}
Let $\pi : T\to K$ be the orientable double covering of $K$ and let  $\pi_1(T)=\left\langle a,b \,\left\lvert\, [a,b] \right.\right\rangle$, so that
$ \pi_{\#}(a)=\alpha$ and $\pi_{\#}(b)=\beta^2$. Given a map $f:K \to K$ such that the induced homomorphism on the fundamental group is given by
$f_{\#}(\alpha)=\alpha^r$,   $f_{\#}(\beta)=\alpha^u\beta^v$, if $v$ is odd, then $f$ lifts to a map from $T$ to $T$, and there are exactly two lifts. The matrices of the induced homomorphisms of these lifts on $\pi_{1}(T)$ are 
$\left(\begin{smallmatrix}
           r & 0 \\
           0 & v
         \end{smallmatrix}\right)$ and
$\left(\begin{smallmatrix}
           -r & 0 \\
           0 & v
         \end{smallmatrix}\right)$.
Observe that the determinant of one of these two matrices is positive. Identifying $MCG(T)$ with $SL(2,\mathbb{Z})$, we conclude that $\phi$ sends the first (resp.\ second) two automorphisms of equation~(\ref{eq:eltsmcgK}) to the matrix $\left( \begin{smallmatrix}1 & 0\\
0 & 1
\end{smallmatrix} \right)$ (resp.\ to $\left( \begin{smallmatrix}-1 & 0\\
0 & -1
\end{smallmatrix} \right)$). In particular, the second part of Proposition~\ref{prop:hopetillmann} does not hold in this case.
\end{remark}

\section{Embeddings of mapping class groups}\label{sec:embeddings}

As in all of this paper, the surfaces $N$ and $S$ under consideration are compact and without boundary, and $\pi: S \to N$ is the double covering defined in Section~\ref{sec:odc}. 
If $X$ is a finite $k$-point subset of $N$, $\til{X}=\pi^{-1}(X)$ and $f\in \topp(N;X)$, then we define a map $\rho_k:\topp(N,X)\lra \topp^{+}(S,\til{X})$ by $\rho_{k}(f)=\rho(f)$, where we consider $f$ to be an element of $\topp(N)$ 
and $\rho$ is as defined in Section~\ref{sec:odc}. 
By Proposition~\ref{prop:hopetillmann}, the map $\rho_{k}$ induces a homomorphism $\phi_{k}:MCG(N;X)\longrightarrow MCG(S;\til{X})$ defined by $\phi_{k}([f])=[\rho_{k}(f)]$.  On the other hand, the map $\psi_{k}: \topp(N) \ra D_{k}(N)$ (resp.\ $\til{\psi}_{2k}: \topp^{+}(S) \ra D_{2k}(S)$) defined by $\psi_{k}(f)=f(X)$ for all $f\in \topp(N)$ (resp.\ $\til{\psi}_{2k}(h)=h(\til{X})$ for all $h\in \topp^{+}(S)$) is a locally-trivial fibration~\cite{Birman69,McCarty63} whose fibre is $\topp(N;k)$ (resp.\ $\topp^+(S;2k)$). The long exact sequence in homotopy of these fibrations gives rise to 
the following exact sequences:
\begin{equation*}
B_k(N)\stac{}MCG(N;X) \stac{\tau_{k}} MCG(N),\;\;B_k(S)\stac{}MCG(S;\til{X}) \stac{\til{\tau}_{k}} MCG(S),
\end{equation*}
where $\tau_{k},\; \til{\tau}_{k}$ are induced by suppressing the markings on $X$ and $\til{X}$ respectively~\cite{Birman69}. In Theorem~\ref{th:braidmcg}, we will see that the injectivity of $\phi$ given by Proposition~\ref{prop:hopetillmann}
carries over to that of the homomorphism $\phi_{k}$ between mapping class groups of marked surfaces and holds for \emph{all} $g\geq 1$, in contrast to the non injectivity in the case $g=2$ described in Remark~\ref{rem:notinjective}.


Let $S$ be a compact, connected surface without boundary, let $l\in \mathbb{N}$, and let $X\subset S$ be a subset of cardinality $l$. If $S$ is orientable, let:
\begin{align*}
\topp_{F}^{+}(S;l)&= \{ f\in \topp^{+}(S) \,\left\lvert\, \text{$f(x_{i})=x_{i}$ for all $i=1,\ldots,l$}\right. \}\\
&= \{ f\in \topp^{+}(S;X) \,\left\lvert\, [f]\in PMCG(S;l)\right. \} \; \text{and}\\
\topp^{+}_{\til{F}}(S;2l) &=\left\{ \til{f}\in \topp^{+}(S) \,\left\lvert \text{$\til{f}(\pi^{-1}(x_{j}))=\pi^{-1}(x_{j})$ for all $j=1,\ldots,l$} \right. \right\},
\end{align*}
and if $S$ is non orientable, let 
\begin{align*}
\topp_{F}(S;l)&=\{ f\in \topp(S) \,\left\lvert\, \text{$f(x_{i})=x_{i}$ for all $i=1,\ldots,l$}\right. \}\\
&= \{ f\in \topp(S;X) \,\left\lvert\, [f]\in PMCG(S;l)\right. \}.
\end{align*}
Observe that if $S$ is orientable (resp.\ non orientable), $\topp^{+}_{F}(S;1)=\topp^{+}(S;1)$ (resp.\ $\topp_{F}(S;1)=\topp(S;1)$). Before giving the proof of Theorem~\ref{th:braidmcg}, we state and prove the following two results. 

\begin{proposition}\label{prop:pi1topT}
Let $k\geq 1$, and consider the homomorphism $\til{\psi}_{2k\#}\colon \pi_{1}(\topp^{+}(T)) \ra B_{2k}(T)$ induced by the map $\til{\psi}_{2k}$ defined above. Then $\pi_{1}(\topp^{+}(T))\cong \mathbb{Z}^{2}$, and $\im{\til{\psi}_{2k\#}}=Z(B_{2k}(T))\subset P_{2k}(T)$.
\end{proposition}

\begin{proof}
First, by~\cite[Theorem~2, p.~63]{Hamstrom65}), $\pi_{1}(\topp^{+}(T))\cong \bz^{2}$. Secondly, taking $\topp^{+}(T)$ to be equipped with $\id_{T}$ as its basepoint, a representative loop of an element $\gamma\in \pi_{1}(\topp^{+}(T),\id_{T})$ is a path in $\topp^{+}(T)$ from $\id_{T}$ to itself. It follows from the definition of $\til{\psi}_{2k}$ that $\til{\psi}_{2k\#}(\gamma)\in P_{2k}(T)$. It remains to show that $\im{\til{\psi}_{2k\#}}=Z(B_{2k}(T))$. Since $\im{\til{\psi}_{2k\#}}\subset P_{2k}(T)$, the homomorphism $\til{\psi}_{2k\#}$ coincides with the homomorphism of~\cite[Theorem~1]{Birman69}. Using the exact sequence of that theorem and~\cite[Corollary~1.3]{Birman69}, we see that $\im{\til{\psi}_{2k\#}}=\left\langle a_{1},b_{1} \right\rangle$, where $a_{1}$ and $b_{1}$ are the generators of $P_{2k}(T)$ defined in~\cite[Theorem~5]{Birman69a}. But by~\cite[Proposition~4.2]{PR}, these two elements generate the centre of $B_{2k}(T)$ as required.
\end{proof}

The first part of the following lemma generalises results of~\cite{Hamstrom64}.
\begin{lemma}\label{lem:hamstrom}\mbox{}
\begin{enumerate}[(a)]
\item\label{it:hamstroma} Let $S$ be a compact, connected orientable (resp.\ non-orientable) surface without boundary for which $\pi_{1}(\topp^{+}(S;1))$, (resp.\ $\pi_{1}(\topp(S;1))$) is trivial, and let $l\geq 1$. Then $\pi_{1}(\topp^{+}_{F}(S;l))$, (resp.\ $\pi_{1}(\topp_{F}(S;l))$) is trivial. In particular, $\pi_{1}(\topp^{+}_{F}(T;l))$ (resp.\ $\pi_{1}(\topp_{F}(K;l))$) is trivial for all $l\geq 1$.
\item\label{it:hamstromb} $\pi_{1}(\topp^{+}_{\til{F}}(T;2l))$ is trivial for all $l\geq 1$.
\end{enumerate}
\end{lemma}

\begin{proof}\mbox{}
\begin{enumerate}[(a)]
\item Assume first that $S$ is orientable. 
We prove the result by induction on $l$. If $l=1$ then the result holds by the hypothesis, using the fact that $\topp^{+}_{F}(S;1)=\topp^{+}(S;1)$. Suppose by induction that the result holds for some $l\geq 1$. The map $\topp^{+}_{F}(S;l)\to S\setminus\{ x_{1},\ldots,x_{l} \}$ given by evaluating an element of $\topp^{+}_{F}(S;l)$ at the point $x_{l+1}$ is a fibration with fibre $\topp^{+}_{F}(S;l+1)$. Taking the long exact sequence in homotopy of this fibration and using the fact that $\pi_{2}(S\setminus\{ x_{1},\ldots,x_{l} \})$ is trivial, it follows that the homomorphism $\pi_{1}(\topp^{+}_{F}(S;l+1)) \to \pi_{1}(\topp^{+}_{F}(S;l))$ induced by the inclusion of the fibre in the total space is injective, and since $\pi_{1}(\topp^{+}_{F}(S;l))$ is trivial by the induction hypothesis, $\pi_{1}(\topp^{+}_{F}(S;l+1))$ is too. This proves the first part of the statement if $S$ is orientable. If $S$ is non orientable, it suffices to replace $\topp^{+}_{F}(S;l)$ by $\topp_{F}(S;l)$ in the proof of the orientable case. The second part of the statement is a consequence of the first part and~\cite[Corollary, p.~65]{Hamstrom65} and~\cite[Theorem~4.1]{Hamstrom64}. 
\item Let $l\geq 0$. The map $\topp^{+}_{\til{F}}(T;2l) \to D_{2}(T \setminus \{ x_{1},\ldots, x_{2l}) \}$ that to a homeomorphism $f\in \topp^{+}_{\til{F}}(T;2l)$ associates the set $\{ f(x_{2l+1}), f(x_{2l+2}) \}$ is a fibration whose fibre is $\topp^{+}_{\til{F}}(T;2l+2)$. If $l=0$ then $\topp^{+}_{\til{F}}(T;2l)$ is just $\topp^{+}(T)$. Taking the long exact sequence in homotopy, and using the fact that $D_{2}(T \setminus \{ x_{1},\ldots, x_{2l} \})$ is a $K(\pi,1)$, we obtain the following exact sequence:
\begin{equation}\label{eq:ftilde}
1\to \pi_{1}(\topp^{+}_{\til{F}}(T;2l+2)) \xrightarrow{(\til{q}_{2l+2})_{\#}} \pi_{1}(\topp^{+}_{\til{F}}(T;2l)) \to B_{2}(T \setminus \{ x_{1},\ldots, x_{2l} \}),
\end{equation}
where the homomorphism $(\til{q}_{2l+2})_{\#}$ is induced by the map $\til{q}_{2l+2}\colon \topp^{+}_{\til{F}}(T;2l+2) \to \topp^{+}_{\til{F}}(T;2l)$
that forgets the marking on the last two points. If $\pi_{1}(\topp^{+}_{\til{F}}(T;2l))$ is trivial for some $l\geq 1$ then clearly $\pi_{1}(\topp^{+}_{\til{F}}(T;2l+2))$ is also trivial. So applying induction on $l$, it suffices to prove the result for $l=1$. 
It follows from Proposition~\ref{prop:pi1topT} that the map $\pi_{1}(\topp^{+}_{\til{F}}(T))\to B_{2}(T)$ sends $\pi_{1}(\topp^{+}_{\til{F}}(T))$ isomorphically onto the centre of $B_{2}(T)$. It follows from exactness of~(\ref{eq:ftilde}) that $\pi_{1}(\topp^{+}_{\til{F}}(T;2))$ is trivial, and this completes the proof of the lemma.\qedhere
\end{enumerate}
\end{proof}

\begin{remark}
If $S$ is a surface that satisfies the hypotheses of Lemma~\ref{lem:hamstrom}(\ref{it:hamstroma}) then $S$ is different from the $2$-sphere~\cite[lines~2--3, p.~303]{McCarty63} and the real projective plane~\cite[Theorem~3.1]{Hamstrom64}, so $S$ is an Eilenberg Mac~Lane space of type $K(\pi,1)$.
\end{remark}

\begin{proof}[Proof of Theorem~\ref{th:braidmcg}] Consider the double covering $\pi:S \ra N$, fix a $k$-point subset $X$ of $N$, and let $\til{X}=\pi^{-1}(X)$. By the comments preceding the statement of the theorem, we obtain the following commutative diagram of fibrations:
\begin{equation}\label{eq:topdiag}
\begin{gathered}
\xymatrix{\topp(N;X) \ar[r]\ar[d]^{\rho_{k}} & \topp(N) \ar[r]^{\psi_{k}}\ar[d]^{\rho} & D_{k}(N) \ar[d]^{\varphi_{k}}\\
\topp^+(S;\til{X}) \ar[r] & \topp^{+}(S) \ar[r]^{\til{\psi}_{2k}} & D_{2k}(S),}
\end{gathered}
\end{equation}
where $\rho$ is as defined in Section~\ref{sec:odc}, and $\varphi_{k}: D_{k}(N) \longrightarrow D_{2k}(S)$ is given by equation~(\ref{eq:liftconfig}). The left-hand square clearly commutes because $\rho_{k}$ is the restriction of $\rho$ to $\topp(N;k)$. We claim that the right-hand square also commutes. On the one hand,
$\pi(\til{\psi}_{2k}\circ \rho(f))=\pi\circ \rho(f)(\til{X})= f\circ \pi(\til{X})=f(X)$, using the fact that $\rho(f)$ is a lift of $f$, so $\til{\psi}_{2k}\circ \rho(f)=\rho(f)(\til{X}) \subset \pi^{-1}(f(X))=\varphi_{k}\circ \psi_{k}(f)$. Conversely, if $y\in \pi^{-1}(f(X))$, then there exists $x\in X$ such that $\pi(y)=f(x)$. If $\til{x}\in \til{X}$ is a lift of $x$ then $y\in \{\rho(f)(\til{x}), J\circ \rho(f)(\til{x})\}=\{\rho(f)(\til{x}), \rho(f) (J(\til{x}))\}\subset \rho(f)(\til{X})$, which proves the claim. Let $\Phi_{1}\colon \pi_{1}\topp{(N)} \to \pi_{1}\topp^{+}(S)$ denote the homomorphism induced by $\rho$ on the level of fundamental groups. We now take the long exact sequence in homotopy of the commutative diagram~(\ref{eq:topdiag}). The form of the resulting commutative diagram depends on the genus $g$ of $N$, hence we consider the following three cases.
\begin{enumerate}[(a)]
\item Suppose that $g\geq 3$. Then by~\cite{Hamstrom66}, 
$\pi_1 \topp(N)$ and $\pi_1 \topp^+(S)$ are trivial, and we obtain the commutative diagram of~(\ref{eq:mcgdiag}). Since $\varphi_{k\#}$ (resp.\ $\phi$) is injective by~\cite{GoncalvesGuaschi12} (resp.\ by Proposition~\ref{prop:hopetillmann}), the injectivity of $\phi_{k}$ follows from the 5-Lemma.

\item Suppose that $g=1$. Using~\cite{GoncalvesGuaschi12,Scott}, we obtain the following commutative diagram of short exact sequences:
\begin{equation}\label{eq:fnpnp2}
\begin{tikzcd}
1 \arrow{r} &\underbrace{\mathbb{Z}/2\mathbb{Z}}_{\pi_1 \topp(\rp{2})} \arrow{r} \arrow{d}{\Phi_{1}} & B_k(\rp{2}) \arrow{r}\arrow{d}{\varphi_{k\#}} & MCG(\rp{2};k) \arrow{r} \arrow{d}{\phi_k} & 1\\
1 \arrow{r} &\underbrace{\mathbb{Z}/2\mathbb{Z}}_{\pi_1 \topp^+(\stwo)} \arrow{r} & B_{2k}(\stwo) \arrow{r} & MCG(\stwo ;2k) \arrow{r} & 1,
\end{tikzcd}
\end{equation}
where in both cases, $\bz/2\bz$ is identified with the subgroup generated by the full-twist braid of the corresponding braid group. Since $\varphi_{k\#}$ is injective by~\cite{GoncalvesGuaschi12}, so is $\Phi_{1}$, and a routine diagram-chasing argument shows that $\phi_{k}$ is injective.

\item Finally suppose that $g=2$, so that $N$ is the Klein bottle $K$ and $S$ is the torus $T$. Let $k\geq 1$. We claim that $\ker{\phi_{k}}\subset PMCG(K;k)$. To prove this, let $X=\{ x_{1}, \ldots, x_{k}\}$.
Let $\sigma\colon MCG(K;k) \to \Sigma_{k}$ and $\til{\sigma}\colon MCG(T;2k) \to \Sigma_{2k}$ denote the usual homomorphisms onto the symmetric groups of $X$ and $\til{X}$ respectively as described in equation~(\ref{eq:sigmamcg}). From the geometric construction of $\rho_{k}$, if $1\leq i,j\leq k$ and $\sigma([f])(x_{i})=x_{j}$ then $\til{\sigma}([\rho_{k}(f)])(y_{i}) \in \pi^{-1}(x_{j})$. If $[f]\in \ker{\phi_{k}}$, where $f\in \topp(K,X)$, then $[\rho_{k}(f)]=\phi_{k}([f])=[\operatorname{\text{Id}}_{S}]$. In particular, $\til{\sigma}([\rho_{k}(f)])= \operatorname{\text{Id}}_{S_{2k}}$, so for all $1\leq i\leq k$, $\til{\sigma}([\rho_{k}(f)])(y_{i}) =y_{i}$, and it follows that $\sigma([f])(x_{i})=x_{i}$, whence $[f]\in PMCG(K;k)$ as claimed. It thus suffices to prove that the restriction $\phi_{k}\left\lvert_{PMCG(K;k)}\right.$ of $\phi_{k}$ to $PMCG(K;k)$ is injective. 
Now assume that $k\geq 2$. Let $p_{k}\colon F_{k}(K)\to F_{k-1}(K)$ (resp.\ $q_{k}\colon \topp_{F}(K;k) \to \topp_{F}(K;k-1)$) be the map given by forgetting the last coordinate (resp.\ the marking on the last point). Let $D_{2k}^{(2)}(T)= F_{2k}(T)/\underbrace{\mathbb{Z}_{2} \times\cdots\times \mathbb{Z}_{2}}_{\text{$k$ times}}$, let $B_{2k}^{(2)}(T)=\pi_{1}D_{2k}^{(2)}(T)$.
Let $\til{p}_{2k}\colon D_{2k}^{(2)}(T)\to D_{2(k-1)}^{(2)}(T)$ be the map given by forgetting the last two coordinates, and let $\til{q}_{2k}\colon \topp^{+}_{\til{F}}(T;2k) \to \topp^{+}_{\til{F}}(T;2(k-1))$ be the map defined in the proof of Lemma~\ref{lem:hamstrom}(\ref{it:hamstromb}).
Then we have the following commutative diagram whose rows are fibrations:
\begin{equation}\label{eq:bigcommdiag}
\begin{gathered}
\xymatrix{%
\topp^{+}_{\til{F}}(T;2(k-1)) \ar[r] & \topp^{+}(T) \ar[r]^{\til{\delta}_{2(k-1)}} & D_{2(k-1)}^{(2)}(T)\\
\topp^{+}_{\til{F}}(T;2k) \ar[r] \ar[u]_{\til{q}_{2k}} & \topp^{+}(T) \ar@{=}[u] \ar[r]^{\til{\delta}_{2k}} & D_{2k}^{(2)}(T) \ar[u]_{\til{p}_{2k}}\\
\topp_{F}(K;k) \ar[r] \ar[u]_{\rho_{k}\left\lvert_{\topp_{F}(K;k)}\right.} \ar[d]^{q_{k}} & \topp(K) \ar[u]_{\rho} \ar@{=}[d] \ar[r]^{\delta_{k}} & F_{k}(K) \ar[u]_{\widehat{\varphi}_{k}} \ar[d]^{p_{k}}\\
\topp_{F}(K;k-1) \ar[r] \ar[d]^{\rho_{k-1}\left\lvert_{\topp_{F}(K;k-1)}\right.} & \topp(K) \ar[d]^{\rho} \ar[r]^{\delta_{k-1}} & F_{k-1}(K)\ar[d]^{\widehat{\varphi}_{k-1}} \\
\topp^{+}_{\til{F}}(T;2(k-1)) \ar[r] & \topp^{+}(T) \ar[r]^{\til{\delta}_{2(k-1)}} & D_{2(k-1)}^{(2)}(T).
}
\end{gathered}
\end{equation}
The map $\delta_{k}\colon \topp(K)\to F_{k}(K)$ (resp.\ $\til{\delta}_{2k}\colon \topp^{+}(T) \to D_{2k}^{(2)}(T)$) is defined by:
\begin{equation*}
\text{$\delta_{k}(f)= (f(x_{1}), \ldots, f(x_{k}))$ (resp.\ $\til{\delta}_{2k}(\til{f})=\bigl( \til{f}(\pi^{-1}(x_{1})), \ldots,\til{f}(\pi^{-1}(x_{k}))\bigr)$)} 
\end{equation*}
for all $f\in \topp(K)$ (resp.\ for all $\til{f}\in \topp^{+}(T)$),
and the map $\widehat{\varphi}_{k}\colon F_{k}(K) \to D_{2k}^{(2)}(T)$ is defined by $\widehat{\varphi}_{k}(v_{1},\ldots, v_{k})= (\pi^{-1}(v_{1}), \ldots, \pi^{-1}(v_{k}))$ for all $(v_{1},\ldots, v_{k})\in F_{k}(K)$. Note also that the diagram remains commutative if we identify the corresponding terms of the first and last rows.
Taking the long exact sequence in homotopy of the diagram~(\ref{eq:bigcommdiag}) and applying Lemma~\ref{lem:hamstrom}, we obtain the following commutative diagram whose rows are exact:
\begin{equation}\label{eq:bigcommdiagexact}
\raisebox{-0.5\height}{\includegraphics{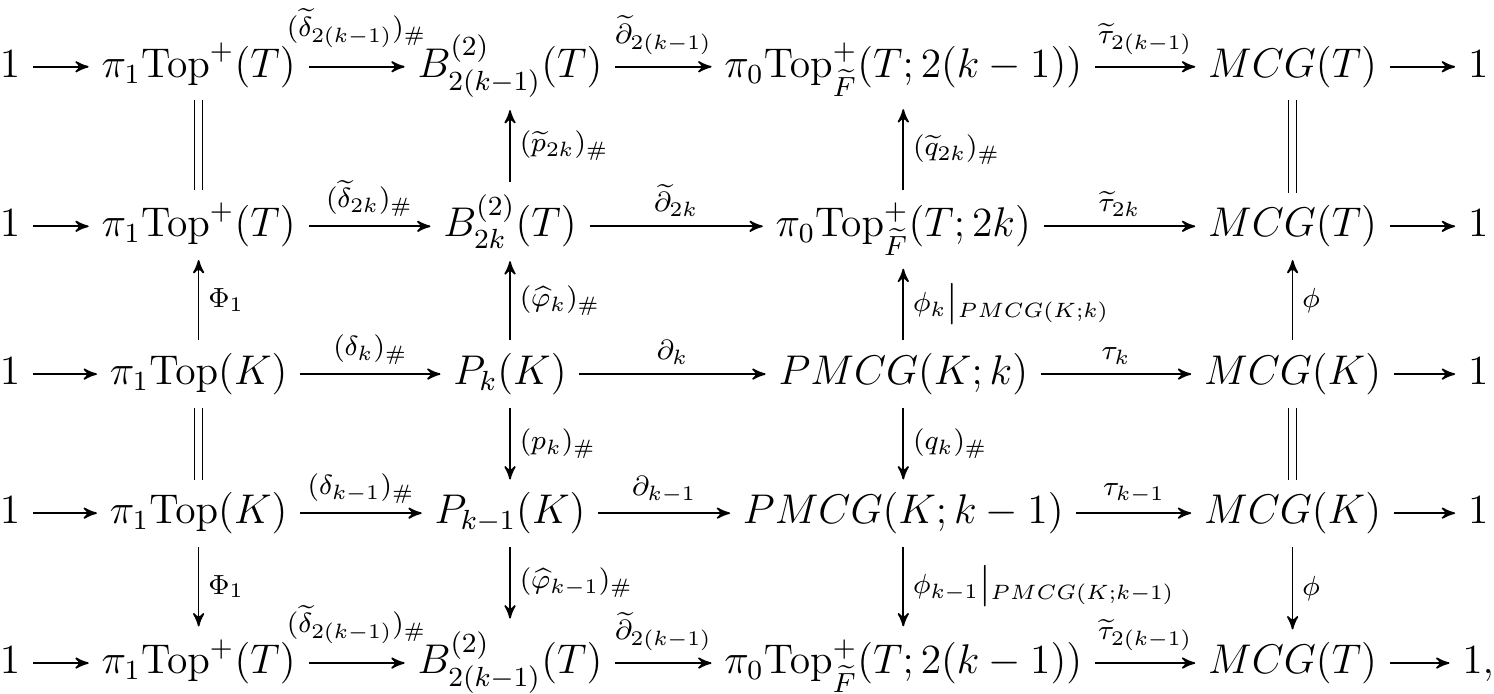}}
\end{equation}
where $\til{\partial}_{2(k-1)},\til{\partial}_{2k},\partial_{k}$ and $\partial_{k-1}$ are the boundary homomorphisms associated with the corresponding fibrations of~(\ref{eq:bigcommdiag}), $\til{\tau}_{2(k-1)}$, $\til{\tau}_{2k}$, $\tau_{k-1}$ and $\tau_{k}$ are the restrictions of these homomorphisms to the given groups, $(\til{q}_{2k})_{\#}$ and $(q_{k})_{\#}$ are the maps induced by $\til{q}_{2k}$ and $q_{k}$ respectively on the level of $\pi_{0}$, and $(\til{\delta}_{2(k-1)})_{\#}$, $(\til{\delta}_{2k})_{\#}$, $(\delta_{k})_{\#}$ and $(\delta_{k-1})_{\#}$ are the homomorphisms induced by $\til{\delta}_{2(k-1)}, \til{\delta}_{2k}, \delta_{k}$ and $\delta_{k-1}$ on the level of $\pi_{1}$. The homomorphism $\phi$ is given by Proposition~\ref{prop:hopetillmann}, and $(\widehat{\varphi}_{k})_{\#}$ is the restriction of $\varphi_{k\#}$ to $P_{k}(K)$, so is injective by~\cite{GoncalvesGuaschi12}.  
We remark that $\pi_{1}\topp(K)\cong \mathbb{Z}$ (resp.\ $\pi_{1}\topp^{+}(T)\cong \mathbb{Z}^{2}$), and may be identified with the centre of $B_{k}(K)$ (resp.\ of $B_{2k}(T)$), which in fact lies in $P_{k}(K)$ (resp.\ in $P_{2k}(T)$).
Once more, the diagram~(\ref{eq:bigcommdiagexact}) remains commutative if we identify the corresponding groups of the first and last rows. In particular, the second and third columns yield respectively:
\begin{align}
(\til{p}_{2k})_{\#}\circ (\widehat{\varphi}_{k})_{\#}&=(\widehat{\varphi}_{k-1})_{\#}\circ (p_{k})_{\#}\label{eq:lastcolumn}\\
(\til{q}_{2k})_{\#} \circ \phi_{k}\left\lvert_{PMCG(K;k)}\right. &=\phi_{k-1}\left\lvert_{PMCG(K;k-1)}\right. \circ (q_{k})_{\#}. \label{eq:thirdcolumn}
\end{align}
We recall at this point that our aim is to show that the homomorphism $\phi_{k}\left\lvert_{PMCG(K;k)}\right.$ is injective. 
Note that by diagram~(\ref{eq:bigcommdiagexact}), $(\widehat{\varphi}_{k})_{\#}\circ (\delta_{k})_{\#}=(\til{\delta}_{2k})_{\#}$, and it follows from the exactness of~(\ref{eq:bigcommdiagexact}) and the injectivity of $(\widehat{\varphi}_{k})_{\#}$ that $\Phi_{1}$ is also injective. 
We claim that the restriction
\begin{equation}\label{eq:iso1}
\til{\partial}_{2k}\left\lvert_{\ker{(\til{p}_{2k})_{\#}}} \right.\colon \ker{(\til{p}_{2k})_{\#}} \to \ker{(\til{q}_{2k})_{\#}}
\end{equation}
is an isomorphism. Indeed:
\begin{enumerate}[(i)]
\item it is well defined. Indeed, if $x\in \ker{(\til{p}_{2k})_{\#}}$ then by the commutative diagram~(\ref{eq:bigcommdiagexact}), $(\til{q}_{2k})_{\#}\circ \til{\partial}_{2k}(x)= \til{\partial}_{2(k-1)}\circ (\til{p}_{2k})_{\#}(x)=1$, so $\til{\partial}_{2k}(x)\in \ker{(\til{q}_{2k})_{\#}}$.

\item suppose that $x\in \ker{\til{\partial}_{2k}\left\lvert_{\ker{(\til{p}_{2k})_{\#}}} \right.}$. By exactness of the second row of the diagram~(\ref{eq:bigcommdiagexact}), there exists $z\in \pi_{1}\topp^{+}(T)$ such that $(\til{\delta}_{2k})_{\#}(z)=x$, and by commutativity of the same diagram, 
 $(\til{\delta}_{2(k-1)})_{\#}(z)= (\til{p}_{2k})_{\#}\circ (\til{\delta}_{2k})_{\#}(z)=(\til{p}_{2k})_{\#}(x)=1$. But $(\til{\delta}_{2(k-1)})_{\#}$ 
 is injective, so $z=1$, whence $x=1$. Thus $\til{\partial}_{2k}\left\lvert_{\ker{(\til{p}_{2k})_{\#}}} \right.$ is injective too.

\item let $y\in \ker{(\til{q}_{2k})_{\#}}$. Using commutativity and exactness of the second row of the diagram~(\ref{eq:bigcommdiagexact}), we have $y\in \ker{\til{\tau}_{2k}}$, so there exists $w\in B_{2k}^{(2)}(T)$ such that $\til{\partial}_{2k}(w)=y$, and thus $(\til{p}_{2k})_{\#}(w)\in \ker{\til{\partial}_{2(k-1)}}$. Hence there exists $z\in \pi_{1}\topp^{+}(T)$ such that $(\til{p}_{2k})_{\#}(w)=(\til{\delta}_{2(k-1)})_{\#}(z)=(\til{p}_{2k})_{\#}\circ (\til{\delta}_{2k})_{\#}(z)$. It follows that $w (\til{\delta}_{2k})_{\#}(z^{-1})\in \ker{(\til{p}_{2k})_{\#}}$, and that $\til{\partial}_{2k}(w (\til{\delta}_{2k})_{\#}(z^{-1}))= \til{\partial}_{2k}(w)=y$ by exactness of the second row of~(\ref{eq:bigcommdiagexact}), which proves that $\til{\partial}_{2k}\left\lvert_{\ker{(\til{p}_{2k})_{\#}}} \right.$ is surjective.
\end{enumerate}
In a similar manner, one may show that the restriction
\begin{equation}\label{eq:iso2}
\partial_{k}\left\lvert_{\ker{(p_{k})_{\#}}} \right.\colon \ker{(p_{k})_{\#}} \to \ker{(q_{k})_{\#}}
\end{equation}
is also an isomorphism,
and if $w\in \ker{(q_{k})_{\#}}$ then $(\til{q}_{2k})_{\#}(\phi_{k}(w))= \phi_{k-1}((q_{k})_{\#}(w))=1$ by equation~(\ref{eq:thirdcolumn}), and so $\phi_{k}(w)\in \ker{(\til{q}_{2k})_{\#}}$. We thus obtain the following commutative diagram of short exact sequences:
\begin{equation}\label{eq:mcgdiag4}
\begin{gathered}
\xymatrix@1{%
1\ar[r]&\ker{(q_{k})_{\#}}\ar[r] \ar[d]^{\phi_{k}\left\lvert_{\ker{(q_{k})_{\#}}}\right.} & PMCG(K;k) \ar[r]^(0.45){(q_{k})_{\#}} \ar[d]^{\phi_{k}\left\lvert_{PMCG(K;k)}\right.} & PMCG(K;k-1) \ar[r] \ar[d]^{\phi_{k-1}\left\lvert_{PMCG(K;k-1)}\right.} & 1\\
1\ar[r]&\ker{(\til{q}_{2k})_{\#}}\ar[r] & \pi_{0}\topp^{+}_{\til{F}}(T;2k) \ar[r]_(0.45){(\til{q}_{2k})_{\#}} & \pi_{0}\topp^{+}_{\til{F}}(T;2(k-1)) \ar[r] & 1,}
\end{gathered}
\end{equation}
the second arrow in each row being inclusion. We claim that $\phi_{k}\left\lvert_{\ker{(q_{k})_{\#}}}\right.$ is injective. This being the case, if $\phi_{k-1}\left\lvert_{PMCG(K;k-1)}\right.$ is injective then $\phi_{k}\left\lvert_{PMCG(K;k)}\right.$ is also injective by the commutativity and exactness of the rows of~(\ref{eq:mcgdiag4}). By induction on $k$, to complete the proof in the case $g=2$, it will thus suffice to prove that the homomorphism $\phi_{1} \colon MCG(K;1) \to MCG(T;2)$ is injective, which we shall do shortly  (note that $PMCG(K;1)=MCG(K;1)$, so we may remove the restriction symbol from $\phi_{1}$). 
We first prove the claim. Let $y\in \ker{(q_{k})_{\#}}$ be such that $\phi_{k}(y)=1$. From the isomorphism~(\ref{eq:iso2}), there exists a unique $x\in \ker{(p_{k})_{\#}}$ such that $\partial_{k}(x)=y$. By~(\ref{eq:lastcolumn}), we have $(\til{p}_{2k})_{\#}\circ (\widehat{\varphi}_{k})_{\#}(x)= (\widehat{\varphi}_{k-1})_{\#} \circ (p_{k})_{\#}(x)=1$, so $(\widehat{\varphi}_{k})_{\#}(x)\in \ker{(\til{p}_{2k})_{\#}}$. On the other hand, $\til{\partial}_{2k}\circ (\widehat{\varphi}_{k})_{\#} (x)=\phi_{k}\circ \partial_{k}(x)=\phi_{k}(y)=1$ by commutativity of the diagram~(\ref{eq:bigcommdiagexact}), hence $(\widehat{\varphi}_{k})_{\#} (x)\in \ker{\til{\partial}_{2k}}$. So $(\widehat{\varphi}_{k})_{\#} (x)=1$ by the isomorphism~(\ref{eq:iso1}). But $(\widehat{\varphi}_{k})_{\#}$ is the restriction of $\varphi_{k\#}$ to $P_{k}(K)$, so is injective~\cite{GoncalvesGuaschi12}. Thus $x=1$, $y=1$, and therefore $\phi_{k}\left\lvert_{\ker{(q_{k})_{\#}}}\right.$ is injective as claimed.

It thus remains to show that the homomorphism $\phi_{1} \colon MCG(K;1) \to MCG(T;2)$ is injective. Suppose that $w \in \ker{\phi_{1}}$. Taking $k=1$, the second and third rows of diagram~(\ref{eq:bigcommdiagexact}) become:
\begin{equation}\label{eq:mcgdiag2}
\begin{gathered}
\xymatrix@1{1\ar[r] & Z(P_1(K)) \ar[r]^(0.55){(\delta_{1})_{\#}} \ar[d]^{\Phi_{1}} & P_1(K)\ar[r]^(.38){\partial_{1}} \ar[d]^{\varphi_{1\#}} & MCG(K;1)\ar[r]^(.54){\tau_{1}}\ar[d]^{\phi_{1}} & {\underbrace{MCG(K)}_{\mathbb{Z}_{2}\oplus \mathbb{Z}_{2}}} \ar[d]^(.6){\phi}\ar[r]&1\\ 
1\ar[r]& Z(P_2(T)) \ar[r]^(0.55){(\til{\delta}_{2})_{\#}} & B_{2}(T)\ar[r]^(.38){\til{\partial}_{2}} & MCG(T;2)\ar[r]^(.52){\til{\tau}_{2}} &
{\underbrace{MCG(T)}_{\operatorname{\text{SL}}(2,\mathbb{Z})}} \ar[r] &1,}
\end{gathered}
\end{equation}
Since $\phi\circ \tau_{1}(w)=\til{\tau}_{2}\circ \phi_{1}(w)=1$, we see that $\tau_{1}(w)\in \ker{\phi}$, and so by Remark~\ref{rem:notinjective}, $\tau_{1}(w)$ is either equal to the mapping class of the identity of $MCG(K)$, in which case we say that $w$ is of \emph{type one}, or is equal to the mapping class of the automorphism 
$\begin{cases}
\alpha \longmapsto \alpha\\
\beta \longmapsto \alpha \beta,
\end{cases}$
in which case we say that $w$ is of \emph{type two}. Recall that $P_1(K)=\left\langle \alpha,  \beta \,\left\lvert\, \alpha \beta\alpha\beta^{-1}\right.\right\rangle$ and $Z(P_1(K))=\mathbb{Z}(\beta^2)$, and every element of $P_1(K)$ may be written (uniquely) in the form $\alpha^{r}\beta^{s}$, where $r,s\in \mathbb{Z}$. By the Appendix, $B_{2}(T)=\left\langle x,y,a,b,\sigma \right\rangle$ and $Z(P_1(T))=\mathbb{Z}(a) \oplus \mathbb{Z}(b)$, where the given generators of $B_{2}(T)$ are subject to the relations of Theorem~\ref{presfull}. The group $B_{2}(T)$ contains
\begin{equation}\label{eq:decompP2T}
P_{2}(T)=\mathbb{F}_2(x,y) \oplus\mathbb{Z}(a)\oplus\mathbb{Z}(b)
\end{equation}
as an index two subgroup, where $\mathbb{F}_2(x,y)$ is the free group on $\{ x,y\}$. The homomorphism $\varphi_{1\#}$  is given by  $\varphi_{1\#}(\alpha)=a^{-1}x^{2}$ and $\varphi_{1\#}(\beta)=y \sigma^{-1}$. Using the relations of Theorem~\ref{presfull}, one may check that $\varphi_{1\#}(\beta^{2})=b$.

First assume that $w\in \ker{\phi_1}$ is of type one.
Then $w\in \ker{\tau_{1}}$, and by exactness of the first row of~(\ref{eq:mcgdiag2}), there exists $w'\in P_1(K)$ such that $\partial_{1}(w')=w$, and so $w'=\alpha^{r}\beta^{s}$, where $r,s\in \mathbb{Z}$ are unique. If $w'\in \left\langle\beta^{2} \right\rangle$ then $\partial_{1}(w')=1$ by exactness of the first row of~(\ref{eq:mcgdiag2}), and the conclusion clearly holds. So suppose that either $s$ is odd or $r\neq 0$. Then
\begin{equation}\label{eq:phiwprime}
\varphi_{1\#}(w')=(a^{-1}x^{2})^{r} (y \sigma^{-1})^{s}=\begin{cases}
x^{2r} a^{-r} b^{s/2} & \text{if $s$ is even}\\
x^{2r} y a^{-r} b^{(s-1)/2}  \sigma^{-1} & \text{if $s$ is odd,}
\end{cases}
\end{equation}
using the fact that $Z(B_{2}(T))=\left\langle a,b \right\rangle$. Since $(\til{\delta}_{2})_{\#}\circ \varphi_{1\#}(w')= \phi_{1}\circ \partial_{1}(w')= \phi_{1}(w)=1$, by observing the induced permutation of $\varphi_{1\#}(w')$, we conclude that $s$ must be even, so $\varphi_{1\#}(w') \in P_{2}(T)$, and that $\varphi_{1\#}(w')\in \left\langle a,b \right\rangle$ by exactness of the second row of~(\ref{eq:mcgdiag2}). Using the decomposition~(\ref{eq:decompP2T}), it follows that $r=0$, which yields a contradiction. So if $w\in \ker{\phi_1}$ is of type one then $w=1$. Now suppose that $w\in \ker{\phi_1}$ is of type two.
Consider the basepoint-preserving homeomorphism $h$ of $K$ illustrated in Figure~\ref{fig:homeoh}. Observe that $\tau_{1}([h])=\tau_{1}(w)$, so by exactness of the first row of~(\ref{eq:mcgdiag2}), there exists $w'\in P_{1}(K)$ such that $w=[h]\,\partial_{1}(w')$. Let $r,s\in \mathbb{Z}$ be such that $w'=\alpha^{r}\beta^{s}$.
 \begin{figure}[h]
\hfill
\begin{tikzpicture}[>=stealth',semithick,auto]
      \draw[->] (0,0) -- (1,0);
      \draw (1,0) -- (2,0);
       \draw[->] (0,2) -- (1,2);
      \draw (1,2) -- (2,2);
      \draw[->] (0,0) -- (0,1);
      \draw (0,1) -- (0,2);
      \draw[->] (2,2) -- (2,1);
      \draw (2,1) -- (2,0);
      \path (2,1) node (a) [right] {$\alpha$}
   (1,0) node (b) [below] {$\beta$}
            (1,2) node (c) [above] {$\beta$}
            (0,1) node (d) [left] {$\alpha$};
     \end{tikzpicture}
  \raisebox{8.5ex}{$\quad \xrightarrow{\phantom{hhhh} h \phantom{hhhh}}$\quad}
\begin{tikzpicture}[>=stealth']
      \draw[->] (0,0) -- (1,0);
      \draw (1,0) -- (2,0);
       \draw[->] (0,2) -- (1,2);
      \draw (1,2) -- (2,2);
      \draw[->] (0,0) -- (0,1);
      \draw (0,1) -- (0,2);
      \draw[->] (2,2) -- (2,1);
      \draw (2,1) -- (2,0);
      \draw (0,0) -- (2,2);
      \draw[->] (0,0) -- (1,1);
      \draw (1,1) -- (2,2);
      \path (2,1) node (a) [right] {$\alpha$}
   (1,0) node (b) [below] {$\beta$}
            (1,2) node (c) [above] {$\beta$}
            (0,1) node (d) [left] {$h(\alpha)=\alpha$}
            (1.25,1.4) node (e) [left] {$h(\beta)$};
    \end{tikzpicture}\hspace*{\fill}
\caption{The homeomorphism $h$ of $K$.}
\label{fig:homeoh}
\end{figure}
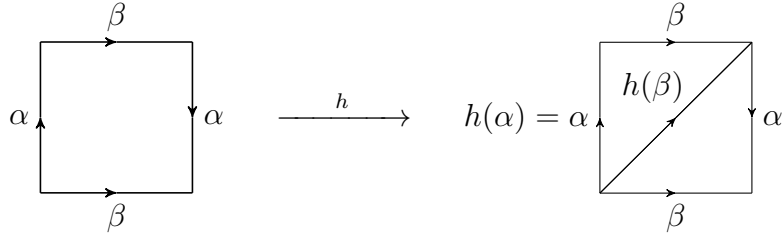
Now $h$ lifts to an orientation-preserving homeomorphism $\til{h}\in \topp^{+}(T,\til{X})$ that is illustrated in Figure~\ref{fig:lifthomeoh}, where $\til{X}=\{ \til{x}_{1}, \til{x}_{2}\}$ is the inverse image of the basepoint of $K$ under $\pi$. It follows that $\til{h}=\rho_{1}(h)$, and hence $[\til{h}]=[\rho_{1}(h)]=\phi_{1}([h])$.
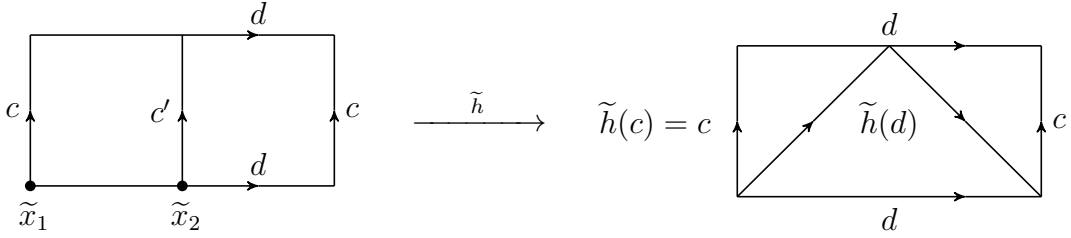
\begin{figure}[h]
\hfill
     \begin{tikzpicture}[>=stealth',semithick,auto]
      \draw[->] (0,0) -- (3,0);
      \draw (3,0) -- (4,0);
       \draw[->] (0,2) -- (3,2);
      \draw (3,2) -- (4,2);
      \draw[->] (0,0) -- (0,1);
      \draw (0,1) -- (0,2);
      \draw[->] (4,0) -- (4,1);
      \draw (4,1) -- (4,2);
            \draw[->] (2,0) -- (2,1);
      \draw (2,1) -- (2,2);
\draw[fill] (0,0) circle [radius=0.06];
\draw[fill] (2,0) circle [radius=0.06];
      \path (4,1) node (a) [right] {$c$}
   (3,0) node (b) [above] {$d$}
            (3,2) node (c) [above] {$d$}
            (0,1) node (d) [left] {$c$}
            (2,1) node (g) [left] {$c'$}
            (0.4,-0.4) node (e) [left] {$\til{x}_{1}$}
            (2.4,-0.4) node (f) [left] {$\til{x}_{2}$};
     \end{tikzpicture}
\raisebox{8ex}{\quad$\xrightarrow{\phantom{hhhh} \til{h} \phantom{hhhh}}$\quad}
    \begin{tikzpicture}[>=stealth',semithick,auto]
      \draw[->] (0,0) -- (3,0);
      \draw (3,0) -- (4,0);
       \draw[->] (0,2) -- (3,2);
      \draw (3,2) -- (4,2);
      \draw[->] (0,0) -- (0,1);
      \draw (0,1) -- (0,2);
      \draw[->] (4,0) -- (4,1);
      \draw (4,1) -- (4,2);
      \draw[->] (0,0) -- (1,1);
      \draw (1,1) -- (2,2);
      \draw[->] (2,2) -- (3,1);
      \draw (3,1) -- (4,0);
      \path (4,1) node (a) [right] {$c$}
   (2,0) node (b) [below] {$d$}
            (2,2) node (c) [above] {$d$}
            (-0.2,1) node (d) [left] {$\til{h}(c)=c$}
            (2,1) node (e) {$\til{h}(d)$};
     \end{tikzpicture}
\hspace*{\fill}
\caption{A lift $\til{h}$ to $T$ of the homeomorphism $h$.}
\label{fig:lifthomeoh}
\end{figure}
 On the other hand, relative to the chosen generators of $P_{2}(T)$, the generator $x$ is represented by the loop $c$ based at $\til{x}_{1}$, and $a$ is represented by two loops, namely $c$ and a parallel copy $c'$ of $c$ based at $\til{x}_{2}$. The loop $c'$ represents the element $x^{-1}a$ of $P_{2}(T)$. Using the definition of $\til{\partial}_{2}$, we must find a homotopy $H$ of $T$ starting at the identity, \emph{i.e.} $H(\cdot , 0)=\operatorname{\text{Id}}_{T}$, for which the evaluation of the homotopy at the points $x_1$ and $x_2$ yields the braid $x^{-1}a$. Let $H$ be the homotopy obtained by pushing $\til{x}_{2}$ along the loop $c'$ and that satisfies $H(\cdot , 1)=\til{h}$. This homotopy clearly has the desired trace (or evaluation), and so $\til{\partial}_{2}(x^{-1}a)=[\til{h}]$. Since $w\in \ker{\phi_{1}}$, we have:
\begin{equation*}
1= \phi_{1}(w)= \phi_{1}([h]\,\partial_{1}(w'))= [\til{h}] \, \til{\partial}_{2}\circ \varphi_{1\#}(w')= \til{\partial}_{2}(x^{-1}a \varphi_{1\#}(w')),
\end{equation*}
by commutativity of the diagram~(\ref{eq:mcgdiag2}), where $\varphi_{1\#}(w')$ is given by equation~(\ref{eq:phiwprime}). Once more, by considering the induced permutation of $x^{-1}a \varphi_{1\#}(w')$, $s$ must be even, and so $\til{\partial}_{2}(x^{2r-1} a^{-r+1} b^{s/2})=1$, where $x^{2r-1} a^{-r+1} b^{s/2}\in P_{2}(T)$ is in the normal form of equation~(\ref{eq:decompP2T}). Since $r\in \mathbb{Z}$, we thus obtain a contradiction. We conclude that $\phi_{1}\left\lvert_{MCG(K;1)}\right.$ is injective, and this completes the proof of the theorem.\qedhere
\end{enumerate}
\end{proof}

\section{Cohomological aspects of mapping class groups of punctured surfaces}\label{sec:vcd}

In this section, we study the (virtual) cohomological dimension
of surface braid groups and mapping class groups with marked points, and we prove Theorem~\ref{th:cdbraid} and Proposition~\ref{prop:vcdmap}. 
\begin{proof}[Proof of Theorem~\ref{th:cdbraid}]
Let $g,k\geq 1$, and let $S=S_{g}$ or $N_{g+1}$. Since $P_{k}(S)$ is of (finite) index $k!$ in $B_{k}(S)$, and $F_{k}(S)$ is a finite-dimensional CW-complex and an Eilenberg-Mac~Lane space of type $K(\pi,1)$~\cite{FaN}, the cohomological dimensions of $P_{k}(S)$ and $B_{k}(S)$ are finite and equal, so it suffices to determine $\operatorname{\text{cd}}(P_{k}(S))$. Let us prove by induction on $k$ that $\operatorname{\text{cd}}(P_{k}(S))=k+1$ and that $H^{k+1}(P_{k}(S),\mathbb{Z})\neq 0$.
The result is true if $k=1$ since then $F_{1}(S)=S$, $H^{2}(\pi_{1}(S),\mathbb{Z})\neq 0$ and $\operatorname{\text{cd}}(P_{1}(S))= \operatorname{\text{cd}}(\pi_{1}(S))=2$. So suppose that the induction hypothesis holds for some $k\geq 1$. The Fadell-Neuwirth fibration $p\colon F_{k+1}(S)\to F_{k}(S)$ given by forgetting the last coordinate gives rise to the following short exact sequence of braid groups:
\begin{equation}\label{fngamma}
1\to N \to P_{k+1}(S) \stackrel{p_{\#}}{\to} P_{k}(S) \to 1,
\end{equation}
where $N=\pi_{1}(S\setminus \{ x_{1},\ldots, x_{k}\},x_{k+1})$, $(x_{1},\ldots, x_{k})$ being an element of $F_{k}(S)$, and $p_{\#}$ is defined geometrically by forgetting the last string. Since $S\setminus \{x_1,\ldots,x_k\}$ has the homotopy type of a bouquet of circles, $H^i(S\setminus \{x_1,\ldots,x_k\}), A)$ is trivial for all $i\geq 2$ and for any choice of local coefficients $A$, and $H^1(S\setminus \{x_1,\ldots,x_k\}), \mathbb{Z})\neq 0$, hence $\operatorname{\text{cd}}(N)=1$. By~\cite[Chapter~VIII]{Br}, it follows that $\operatorname{\text{cd}}(P_{k+1}(S))\leq \operatorname{\text{cd}}(P_{k}(S))+\operatorname{\text{cd}}(N) \leq k+2$. To conclude the proof of the theorem, it remains to show that there exist local coefficients $A$ such that $H^{k+2}(P_{k+1}(S), A)\neq 0$. We will show that this is the case for $A=\mathbb{Z}$. By the induction hypothesis, we have $H^{k+1}(P_{k}(S), \mathbb{Z})\neq 0$. Consider the Serre spectral sequence with integral coefficients associated to the fibration $p$. Then we have that:
\begin{equation*}
E^{p,q}_2= H^{p}\bigl(P_{k}(S), H^q(S\setminus \{ x_{1},\ldots, x_{k}\})\bigr).
\end{equation*}
Since $\operatorname{\text{cd}}(P_{k}(S))=k+1$ and $\operatorname{\text{cd}}(S\setminus \{ x_{1},\ldots, x_{k}\})=1$ from above, this spectral sequence has two horizontal lines whose possible non-vanishing terms occur for $0\leq p \leq k+1$ and $0\leq q \leq 1$. We claim that the group $E^{k+1,1}_2$ is non trivial. To see this, first note that $H^1(S\setminus \{x_1,\ldots,x_k\}), \mathbb{Z})$ is isomorphic to the free Abelian group of rank $r$, where $r=2g+k-1$ if $S=S_{g}$ and $r=g+k$ if $S=N_{g+1}$, so $r\geq 2$, and hence $E^{k+1,1}_2=H^{k+1}\bigl(P_{k}(S), \mathbb{Z}^r\bigr)$, where we identify $\mathbb{Z}^r$ with (the dual of) $N\textsuperscript{Ab}$, the Abelianisation of $N$. The action of $P_{k}(S)$ on $N$ by conjugation induces an action of $P_{k}(S)$ on $N\textsuperscript{Ab}$. Let $H$ be the subgroup of $N\textsuperscript{Ab}$ generated by the elements of the form $\alpha(x)-x$, where $\alpha\in P_{k}(S)$ and $x\in N\textsuperscript{Ab}$ and $\alpha(x)$ represents the action of $\alpha$ on $x$. Then we obtain a short exact sequence $0\to H\to N\textsuperscript{Ab} \to N\textsuperscript{Ab}/H\to 0$ of Abelian groups, and the long exact sequence in cohomology applied to $P_{k}(S)$ yields:
\begin{equation}\label{eq:lescohom}
\cdots \to H^{k+1}(P_{k}(S), N\textsuperscript{Ab}) \to H^{k+1}(P_{k}(S), N\textsuperscript{Ab}/H) \to H^{k+2}(P_{k}(S), H)\to \cdots.
\end{equation}
The last term is zero since $\operatorname{\text{cd}}(P_{k}(S))=k+1$, and so the map between the remaining two terms is surjective. So to prove that $E^{k+1,1}_2$ is non trivial, it suffices to show that $H^{k+1}(P_{k}(S), N\textsuperscript{Ab}/H)$ is non trivial. To do so, we first determine $N\textsuperscript{Ab}/H$. Suppose first that $S=S_{g}$, and consider the presentation of $P_{k}(S)$ on $N\textsuperscript{Ab}$ given in~\cite[Corollary~8]{GGorientable}. With the notation of that paper, a basis $\mathcal{B}$ of $N$ is given by:
\begin{equation*}
\left\{ \left. \rho_{k+1,r},\, \tau_{k+1,r},\, C_{i,k+1} \right\rvert 1\leq r\leq g,\, 1\leq i\leq k-1 \right\},
\end{equation*}
and a set of coset representatives in $P_{k+1}(S)$ of a generating set $\mathcal{S}$ of $P_{k}(S)$ is given by $\left\{ \left. \rho_{m,r},\, \tau_{m,r},\, C_{i,j} \right\rvert 1\leq r\leq g,\, 1\leq i< j\leq k,\, 1\leq m\leq k \right\}$. Using~\cite[Table~1 and Theorem~7 or Corollary~8]{GGorientable}, one sees that the commutators of the elements of $\mathcal{S}$ with those of $\mathcal{B}$ project to the trivial element of $N\textsuperscript{Ab}$, with the exception of $[\rho_{m,r}, \tau_{k+1,r}]$ and $[\tau_{m,r}, \rho_{k+1,r}]$ that project to (the coset of) $C_{m,k+1}$ for all $m=1,\ldots,k-1$ (we take the opportunity here to correct a couple of small misprints in these results: relation~(19) of Theorem~7 should read as in relation~($\text{XII}_{n}$) of Corollary~8; and in Table~1, in each of the three rows, the first occurrence of $j>k$ should read $j<k$). Note that we obtain a similar result for $m=k$, but using the surface relation~\cite[equation~(1)]{GGorientable} in $P_{k+1}(S)$ and taking $i=k+1$ yields no new information. It follows that $H$ is the subgroup of $N\textsuperscript{Ab}$ generated by the $C_{m,k+1}$, where $1\leq m\leq k-1$, and that $N\textsuperscript{Ab}/H$ is the free Abelian group generated by the coset representatives of the $\rho_{k+1,r}$ and $\tau_{k+1,r}$, where $1\leq r\leq g$. In particular, $N\textsuperscript{Ab}/H\cong \mathbb{Z}^{2g}$.
Since the induced action of $P_{k}(S)$ on $N\textsuperscript{Ab}/H$ is trivial, using the induction hypothesis, we conclude that:
\begin{equation*}
H^{k+1}(P_{k}(S), N\textsuperscript{Ab}/H)=\bigl( H^{k+1}(P_{k}(S), \mathbb{Z})\bigr)^{2g}\neq 0.
\end{equation*}
It then follows from~(\ref{eq:lescohom}) that $E^{k+1,1}_2=H^{k+1}(P_{k}(S), N\textsuperscript{Ab})\neq 0$. Since  $E^{p,q}_2=0$ for all $p>k+1$ and $q>1$, we have $E^{k+1,1}_2=E^{k+1,1}_{\infty}$, thus $E^{k+1,1}_{\infty}$ is non trivial, and hence $H^{k+2}(P_{k+1}(S), \mathbb{Z})\neq 0$. This proves the result in the orientable case.

Now let us turn to the non-orientable case. The idea of the proof is the same as in the orientable case, but the computations for $N\textsuperscript{Ab}$ and $H$ are a little different. We use the presentation of $P_{k}(S)$ given in~\cite{GGnonorientable}. With the notation of that paper, a basis $\mathcal{B}$ of $N$ is given by $\left\{ \left. \rho_{k+1,r},\, B_{i,k+1} \right\rvert 1\leq r\leq g,\, 1\leq i\leq k-1 \right\}$, and a set of coset representatives in $P_{k+1}(S)$ of a generating set $\mathcal{S}$ of $P_{k}(S)$ is given by $\left\{ \left. \rho_{m,r},\, B_{i,j} \right\rvert 1\leq r\leq g,\, 1\leq i< j\leq k,\, 1\leq m\leq k \right\}$. Using~\cite[Theorem~3]{GGnonorientable}, one sees that the commutators of the elements of $\mathcal{S}$ with those of $\mathcal{B}$ project to the trivial element of $N\textsuperscript{Ab}$, with the exception of $[\rho_{m,r}, \rho_{k+1,r}]$ that projects to (the coset of) $B_{m,k+1}^{-1}$ for all $m=1,\ldots,k-1$ and $1\leq r \leq g$. We obtain a similar result for $m=k$, and using the surface relation~\cite[relation~(c), Theorem~3]{GGnonorientable} in $P_{k+1}(S)$ and taking $i=k+1$ implies that the element $2\sum_{l=1}^{g} \rho_{k+1,l}$ also belongs to $H$ using additive notation. It follows that $H$ is the subgroup of $N\textsuperscript{Ab}$ generated by the $B_{m,k+1}$, where $1\leq m\leq k-1$, and the element $2\sum_{l=1}^{g} \rho_{k+1,l}$. We conclude that $N\textsuperscript{Ab}/H\cong \mathbb{Z}^{g-1}\oplus \mathbb{Z}_{2}$. The argument then goes through as in the orientable case. 
\end{proof}

The exact sequences given by~(\ref{eq:mcgdiag}) and~(\ref{eq:bigcommdiagexact}) also allow us to obtain information about the virtual cohomological dimension of $MCG(N_{g};k)$ if $g\geq 2$ and $k\geq 1$.

\begin{proof}[Proof of Proposition~\ref{prop:vcdmap}]\mbox{}
\begin{enumerate}[(a)]
\item   Since $PMCG(K;k)$ is a subgroup of finite index of  $MCG(K)$, it suffices to show  that   
$\operatorname{\text{vcd}}(PMCG(K;k))=k$. From the  exact sequence given by the third row of~(\ref{eq:bigcommdiagexact}):
\begin{equation*}
1 \to  \pi_{1}\topp(K) \xrightarrow{(\delta_{k})_{\#}}   P_{k}(K) \stackrel{\partial_{k}}{\to}   PMCG(K;k)  \stackrel{\tau_{k}}{\to} MCG(K) \to  1,
\end{equation*}
since $MCG(K)$ is finite, $PMCG(K;k)$ has the same vcd as the quotient $P_{k}(K)/Z(P_{k}(K))$, 
where $ \pi_{1}\topp(K)\cong Z(P_{k}(K))$. We now prove the result by induction on $k$. First, the result is 
true for $k=1$ since $P_{1}(K)/Z(P_{1}(K))\cong \mathbb{Z}\oplus\mathbb{Z}_2$. Now suppose that the result holds for some $k\geq 1$. 
Since the projection $P_{k+1}(K) \to  P_{k}(K)$ maps the centre $Z(P_{k+1}(K))$ of $P_{k+1}(K)$ isomorphically onto the centre $Z(P_{k}(K))$ of $P_{k}(K)$, we have the following short exact sequence of groups:
\begin{equation*}
1\to  P_{1}(K\setminus\{x_0\}) \to    P_{k+1}(K)/Z(P_{k+1}(K)) \to P_{k}(K)/Z(P_{k}(K))  \to 1.
\end{equation*}
Now $\operatorname{\text{vcd}}(P_{k}(K)/Z(P_{k}(K)))=k$ by induction, hence $\operatorname{\text{vcd}}(P_{k+1}(K)/Z(P_{k+1}(K))\leq k+1$, because $\operatorname{\text{vcd}}(\mathbb{Z})=1$. Using the short exact sequence 
\begin{equation*}
1\to  Z(P_{k+1}(K)) \to    P_{k+1}(K) \to   P_{k+1}(K)/Z(P_{k+1}(K))   \to 1,
\end{equation*}
and the fact that $\operatorname{\text{cd}}(P_{k+1}(K))=k+2\leq 1+\operatorname{\text{vcd}}(P_{k+1}(K)/Z(P_{k+1}(K))$ by Theorem~\ref{th:cdbraid}, we see that $\operatorname{\text{vcd}}( P_{k+1}(K)/Z(P_{k+1}(K))\geq k+1$, so $\operatorname{\text{vcd}}(P_{k+1}(K)/Z(P_{k+1}(K))= k+1$, and the result follows.
 


\item Let $g\geq 3$. Consider the short exact sequence:
\begin{equation*}
1\to B_k(N_{g})\to  MCG(N_{g};k)\to MCG(N_{g})\to 1 
\end{equation*}
given by equation~(\ref{eq:mcgdiag}). Then by~\cite[Corollary~2.2]{HT09} and Theorem~\ref{th:cdbraid}, we have:
\begin{equation*}
\operatorname{\text{vcd}}(MCG(N_g;k))\leq \operatorname{\text{cd}}(B_k(N_g))+\operatorname{\text{vcd}}(MCG(N_g))\leq k+1+4g-9=4g+k-8
\end{equation*}
as required.\qedhere
\end{enumerate}
\end{proof}

\appendix

\section*{Appendix}

\setcounter{theorem}{0}
\renewcommand{\thetheorem}{A.\arabic{theorem}}

In this appendix, we provide presentations of $P_2(T)$ and $B_2(T)$ that are adapted to our situation. From~\cite[Section 4]{FH}, we have:

\begin{theorem}\label{thm:presP2T}
The group $P_2(T)$ possesses a presentation with generators $ B_{1,2}$, $\rho_{1,1}$, $\rho_{1,2}$, $\rho_{2,1}$ and $\rho_{2,2}$ subject to the following relations:
\begin{enumerate}[(a)]
\item $[\rho_{1,1}, \rho_{1,2}^{-1}]=[\rho_{2,1}, \rho_{2,2}^{-1}]=B_{1,2}$.
\item $ \rho_{2,1}\rho_{1,1} \rho^{-1}_{2,1} =  B_{1,2}\rho_{1,1} B_{1,2}^{-1}$.
\item $ \rho_{2,1}\rho_{1,2}\rho_{2,1}^{-1} = B_ {1,2}\rho_{1,2}[\rho_{1,1}^{-1}, B_{1,2}]$.
\item $ \rho_{2,2}\rho_{1,1} \rho^{-1}_{2,2} = \rho_{1,1} B_{1,2}^{-1}$.
\item $ \rho_{2,2}\rho_{1,2}\rho_{2,2}^{-1} = B_ {1,2}\rho_{1,2}B_{1,2}^{-1}$.
\end{enumerate}
\end{theorem}

If necessary, the generator $B_{1,2}$ may be suppressed from the list of generators of $P_2(T)$. Using Theorem~\ref{thm:presP2T}, we may obtain the following useful relations in $P_2(T)$:
\begin{equation*}
\text{$\rho_{2,1} B_{1,2}\rho_{2,1}^{-1} = B_ {1,2}\rho_{1,1}^{-1}B_{1,2}\rho_{1,1}B_{1,2}^{-1}$ and 
$\rho_{2,2} B_{1,2}\rho_{2,2}^{-1} = B_ {1,2}\rho_{1,2}^{-1}B_{1,2}\rho_{1,2}B_{1,2}^{-1}$.}
\end{equation*}
Setting $\delta_{1,1}=\rho_{1,1}$, $\tau_{1,1}=\rho_{1,2}$, $\delta_{2,1}=B_{1,2}^{-1}\rho_{2,1}$ and $\tau_{2,1}=B_{1,2}^{-1}\rho_{2,2}$, we obtain a new presentation of $P_2(T)$ whose generators are $B_{1,2}$, $\delta_{1,1}$, $\tau_{1,1}$, $\delta_{2,1}$ and $\tau_{2,1}$, that are subject to the following relations:
\begin{enumerate}[(a)]
\item $ [ \delta_{1,1}, \tau^{-1}_{1,1}] = [ B_{1,2}\delta_{2,1}, \tau_{2,1}^{-1}B_{1,2}^{-1}] = B_{1,2}$.
\item $ [\delta_{2,1}, \delta_{1,1}] =[\tau_{2,1}, \tau_{1,1}]= 1$.
\item $\delta_{2,1}\tau_{1,1}\delta^{-1}_{2,1} = \tau_{1,1}\delta^{-1}_{1,1}B_{1,2}\delta_{1,1}$.
\item $\tau_{2,1}\delta_{1,1}\tau^{-1}_{2,1}  = B_{1,2}^{-1}\delta_{1,1}$.
\end{enumerate}
If we let $a=\delta_{1,1}\delta_{2,1}$  $b=\tau_{1,1}\tau_{2,1}$, $x=\delta_{1,1}$ and $y=\tau_{1,1}$, then it is not hard to show that $P_{2}(T)$ has a presentation whose generators are $x$, $y$, $a$ and $b$ that are subject to the following relations:
\begin{enumerate}[(a)]
\item $ [ a, b] =1$.
\item $ [a , x] =[b , x]=1$.
\item $[a , y] =[b,  y]=1$. 
\end{enumerate}
It follows from this presentation that
$P_2(T)$ is isomorphic to $F_2\oplus \mathbb{Z} \oplus \mathbb{Z}$, where $\{x, y\}$ is a basis of the free group $F_2$, and  
$\{a,  b\}$ is a basis of the free Abelian group $\mathbb{Z}\oplus \mathbb{Z}$. We then obtain the following presentation of $B_2(T)$. 

\begin{theorem}\label{presfull}
The group $B_2(T)$ has a presentation with generators $B=B_{1,2}, \sigma, x, y, a$ and $b$ that are subject to the following relations:
\begin{enumerate}[(a)]
\item $\sigma^2=[x,  y^{-1}]=B$.
\item $[a, b^{-1}] =1$.
\item $ axa^{-1} = x$ and $ aya^{-1} = y$.
\item $ bxb ^{-1} = x$ and $ byb^{-1}=y$.
\item $\sigma x \sigma^{-1}=Bx^{-1}a$ and $\sigma y \sigma^{-1}=By^{-1}b$.
\item $\sigma a \sigma^{-1}=a$ and $\sigma b \sigma^{-1}=b$.
\end{enumerate}
\end{theorem}

The generator $\sigma$ given in the statement of Theorem~\ref{presfull} is the Artin generator $\sigma_{1}$. The proof of the theorem is standard, and makes use of the above presentation of $P_{2}(T)$, the short exact sequence:
\begin{equation*}
1\to P_2(T) \to B_2(T) \to \mathbb{Z}_2 \to 1,
\end{equation*}
standard results on presentations of group extensions~\cite[Proposition~1, p.~139]{Jo}, and computations of relations between the given generators by means of geometric arguments.

\end{document}